%% file: rev_main.tex
\documentclass[9pt,shortpaper]{IEEEtran}
\ifCLASSINFOpdf
\else
\fi

\usepackage[pdftex]{graphicx}
\usepackage{color}

\usepackage{times}
\usepackage{bm}
\usepackage{bbm}
\usepackage{tikz}
\usepackage{comment}
\usepackage{mathtools}

\usepackage{booktabs}

\newcommand{\revise}[1]{{\color{black}{#1}}}

\newcommand{\qKL}[2]{D_{q} \left( {#1} \middle\| {#2} \right)}

\newcommand{\supp}[1]{\textup{supp}({#1})}

\input{preambles}
\input{macros}

\begin{document}
%
\title{Tsallis Entropy Regularization for  Linear Quadratic Regulator and Kullback–Leibler control}
%
%
%

\author{Yota Hashizume$^{1}$, Koshi Oishi$^{1}$, Shota Yamanaka$^{2}$, and Kenji Kashima$^{1}$, {\it Senior Member, IEEE} 
\thanks{
$^{1}$Y. Hashizume, K. Oishi, and K. Kashima are with the Graduate School of Informatics,
Kyoto University, Kyoto, Japan.
{\tt\small hashizume.yota.88n@st.kyoto-u.ac.jp; oishi.koshi.34y@st.kyoto-u.ac.jp; kk@i.kyoto-u.ac.jp}

$^{2}$S. Yamanaka is with Frontier Research Center, Toyota Motor Corporation, Aichi, Japan.
{\tt\small shota_yamanaka_aa@mail.toyota.co.jp}

This work was partially supported by JSPS KAKENHI Grant Number JP21H04875 and the joint project of Kyoto University
and Toyota Motor Corporation, titled ``Advanced Mathematical Science for Mobility Society''.}%
}

%
%

\markboth{Journal of \LaTeX\ Class Files,~Vol.~14, No.~8, August~2015}%
{Shell \MakeLowercase{\textit{et al.}}: Bare Demo of IEEEtran.cls for IEEE Journals}
%



\maketitle

\begin{abstract}
Shannon entropy regularization is widely adopted in optimal control due to its ability to promote exploration and enhance robustness, e.g., maximum entropy reinforcement learning known as Soft Actor-Critic. 
\revise{The aim of this paper is to show that formulations based on Tsallis entropy, which is a one-parameter extension of Shannon entropy, retain many of the structural and computational advantages of Shannon-entropy-based approaches while offering additional benefits. In particular, we derive a closed-form solution for the linear quadratic regulator and an efficient computational method for the Kullback–Leibler control problem.
    We also }
    demonstrate its usefulness in balancing between exploration and sparsity of the obtained control law.
\end{abstract}


%
\IEEEpeerreviewmaketitle

\section{Introduction}\label{sec:introduction}
%
%
%
%
The incorporation of Shannon entropy of a control policy into the objective function of reinforcement learning, 
    a technique known as maximum entropy reinforcement learning, has been applied in approaches like Soft Actor-Critic \cite{haarnoja_soft_2018}. 
    This method finds practical applications in various real-world scenarios, notably in fields like robotics.
An important feature of using entropy regularization is that the optimal policy is stochastic which is advantageous for exploration and 
  improves robustness \cite{eysenbach_maximum_2021}.
  \revise{Another example is privacy protection, where input randomization induced by entropy regularization makes it difficult for an eavesdropper to estimate the system state \cite{ito_DP}.}
  Given these practical benefits, there is extensive research on optimal control utilizing entropy regularization.
  When using entropy regularization, 
  the probability density function of the optimal control policy takes positive values for all input values. 
Due to this property, entropy regularization is not suitable for problems where you 
  want to reduce the number of possible input values.

  Regularization using Tsallis entropy allows us to obtain
    a high-entropy but sparse control law.
    The sparsity we aim for in this paper is different from sparsity as used in the context of sparse control \cite{ito_sparse_2021, mishra_sparse_2018}. The latter mainly refers to the length of time the input 
    value is $0$, whereas the sparsity here refers to the small support of the 
    probability distribution of the control law. In other words, increasing the 
    sparsity in this sense means increasing the number of input values for which the 
    probability of realization is $0$; 
    See Section \ref{subsec:qKL-ex} for a practical example.

    In \cite{bao_sparse_2022}, Tsallis entropy, 
    which originates from Tsallis statistical mechanics \cite{tsallis_possible_1988}, is used to regularize optimal transport problems to obtain high-entropy, but sparse solutions.
    When regularization is performed using Shannon entropy, the optimal control law is 
    farthest from sparsity.
    On the other hand, the support of a deterministic control law is a singleton set, 
    so it can be said to be the sparsest probability control law, 
    but in this case, the entropy is $0$. 
    Regularization using Tsallis entropy allows us to obtain a control law that lies 
    between these two extreme properties.
    \revise{This allows the control input range to be limited in physical implementations (while retaining the benefit achieved by randomization), which is of clear practical importance; see Fig.~\ref{fig:q-LQR}. In the context of privacy protection, unbounded noise, which is unnecessary
in most practical situations except for concealing outliers \cite{ito_DP}, can be avoided by employing Tsallis
entropy regularization.
As another example, in logistics and transportation systems, balancing exploration and sparsity allows traffic to avoid excessive concentration on a single route for robustness, while at the same time limiting the number of active routes; see Fig.~\ref{fig:qkl-result}.}
    
    In the context of reinforced learning, regularization using Tsallis entropy has been studied both theoretically and in applications \cite{lee_sparse_2018, chen_effective_2018, lee_generalized_2020}. 
    One of the major issues in Tsallis entropy regularized reinforcement learning is that even if the state value function or action value function can be estimated, 
    calculating the $q$-potential required for normalizing the policy is generally difficult; See Section \ref{subsec:Bellman} below.
    Previous studies have taken the following approaches to address this problem:
    \begin{itemize}
      \item \cite{lee_sparse_2018} limited the entropic index $q$, which is the degree of freedom in Tsallis entropy, to $q=0$ for analytical calculation
      \item \cite{chen_effective_2018} performed calculations by using Taylor expansion and first-order approximation
      \item \cite{lee_generalized_2020} avoided the calculation of $q$-potential by approximating the policy as a normal distribution
    \end{itemize}
    As mentioned above, in general MDP settings, this problem is solved by performing approximate calculations unless the entropic index $q$ is limited. 
    
    In this study, we propose a method to obtain the optimal control law without approximation by performing Tsallis entropy regularization 
    in the range of $0\leq q < 1$ for Kullback–Leibler (KL) control and Linear Quadratic Regulator (LQR), which are common in the context of control and are easier to solve than MDP.
    Through numerical examples, we verify that the optimal control policies achieve 
  high entropy while maintaining sparsity, demonstrating the usefulness of TROC. A preliminary version, in which Theorem \ref{thm:lqr-TROC} and Theorem \ref{thm:qKL} are proven, is uploaded  \cite{hashizume_tsallis_2024}.
  
  The rest of the paper is organized as follows.
  In Section \ref{sec:preliminary}, 
  we describe the definitions and fundamental properties related to Tsallis entropy and Tsallis statistics. 
  In Section \ref{sec:TROC}, we first formulate the TROC and derive its Bellman equation. 
  Subsequently, we defined a problem similar to TROC and discussed the robustness of its solution to provide an outlook on the robustness of TROC's solution in Section \ref{sec:robust}.
  The theorem presented in this section generalizes Shannon entropy-regularized optimal control case, clarifying that increasing the Tsallis entropy coefficient enhances robustness.
  In Sections \ref{sec:qLQR} and \ref{sec:qKL}, based on the results from Section \ref{sec:TROC}, 
  we derive the optimal control policies for TROC applied to KL control 
  and linear quadratic regulators, respectively. 
  Section \ref{sec:conclusion} concludes the paper.

  \noindent{\bf Notation}\hspace{3mm}
  Let $\mathbb{E}[\cdot]$ and $\mathbb{V}[\cdot]$ denote the expected value and variance of a random variable, respectively. 
  When the distinction between a random variable and its realization is not clear, 
  the random variable is denoted by $x$, and its realization by $\bm{x}$.   
    Let $\calP(\bbR^n)$ denote the set of all probability density functions on $\bbR^n$, defined as
    $\calP(\bbR^n) := \left\{ \varphi: \bbR^n \to [0, \infty) \mid \int_{\mathbb{R}^n} \varphi(x) dx = 1 \right\}$.    
  Let $\supp{\varphi}$ denote the support of the density function $\varphi\in \calP(\bbR^n)$, that is, the set $\{x\mid \varphi(x)>0\}$. 
  The Gamma function is denoted by $\Gamma$. 
  Let $\calS_{++}^{n}$ be the set of all 
  positive-definite matrices in ${\mathbb R}^{n\times n}$.

\section{Preliminary: Tsallis entropy}\label{sec:preliminary}
  In the context of Tsallis statistical mechanics \cite{tsallis_possible_1988}, $q$-exponential
  functions, $q$-products, and $q$-sums are defined as one-parameter extensions of
  the usual exponential functions, products, and sums, where $q$ is called a
  entropic index \cite{lee_generalized_2020}.
  In the limit $q \to 1$, they coincide with the usual ones. For simplicity, we assume $0 \leq q < 1$ in this paper.
  \begin{definition}[$q$-Exponential and $q$-Logarithm functions]
    \begin{align}
      \exp_{q}(x) & := \left[1+(1-q)x\right]_{+}^{\frac{1}{1-q}} \quad &x\in\mathbb{R}, \label{eq:q-exp}\\
      \log_{q}(x) & := \frac{x^{1-q}-1}{1-q}  \quad &x>0, \label{eq:q-log}
    \end{align}
    where $[x]_{+}:= \max(x,0)$.
  \end{definition}

  \begin{rmk}
        The inverse function relationship exists, i.e., 
    \begin{align}
      \exp_{q}(\log_{q}(x)) = x,\   & x>0,\label{eq:inv-1} \\
      \log_{q}(\exp_{q}(x)) = x,\  & x>-\frac{1}{1-q}.
    \end{align}
    However, the standard exponent rules do not hold:
    \begin{align}
      \exp_{q}(x+y) & \neq \exp_{q}(x)\exp_{q}(y),    \\
      \log_{q}(xy)  & \neq \log_{q}(x) + \log_{q}(y).
    \end{align}
    although $q$-product can attain a similar formula  \cite{suyari_mathematical_2005}.
  \end{rmk}

  Next, we define Tsallis entropy, which is a generalization of Shannon entropy
  \cite{abe_nonextensive_2001}. It converges to Shannon entropy in the limit
  $q \to 1$.
  \begin{definition}
    Tsallis entropy $\mathcal{T}_{q}:\calP(\bbR^n)\rightarrow \bbR$ is defined as
    \begin{align}
      \mathcal{T}_{q}(\varphi) & := -\frac{1}{q}\left(\int \varphi(x)^{q}\log_{q}\varphi(x) dx - 1\right).
    \end{align}
  \end{definition}

  The deformed $q$-entropy defined below is used as a regularization term in \cite{bao_sparse_2022}.
  It is related to the Tsallis entropy by the following relationship, which is referred
  to as an additive duality \cite{suyari_multiplicative_2008}.
  \begin{proposition}[$q$-Entropy]
    The deformed $q$-entropy $\mathcal{H}_{q}:\calP(\bbR^n)\rightarrow \bbR$ defined as
    \begin{equation}
      \mathcal{H}_{q}(\varphi) := - \frac{1}{2-q}\left(\int \varphi(x)\log_{q}\varphi(x) dx - 1\right)
    \end{equation}
    satisfies
    \begin{align}
      \mathcal{H}_{q}(\varphi) = \mathcal{T}_{2-q}(\varphi).
    \end{align}
  \end{proposition}

  This proposition indicates that the Tsallis entropy and deformed $q$-entropy are
  equivalent. In this study, we will use the deformed $q$-entropy as a regularization
  term for the sake of notational simplicity.

  The KL divergence and Gaussian distribution are extended as follows \cite{furuichi_fundamental_2004,
  vignat_central_2007}:
  \begin{definition}[$q$-KL Divergence]
    The $q$-KL divergence $\qKL{\varphi}{\psi}$ between density functions $\varphi$ and $\psi \in \calP(\bbR^n)$ is defined
    as
    \begin{align}
      \qKL{\varphi}{\psi}:= \frac{1}{2-q}\left(\int \varphi(x) \log_{q}\frac{\varphi(x)}{\psi(x)}dx - 1\right).
    \end{align}
  \end{definition}
  The $q$-KL divergence possesses some \revise{metric-like} properties of the KL divergence, such as
  non-negativity and convexity \cite{furuichi_fundamental_2004}.

  \begin{definition}[multivariate $q$-Gaussian]\label{def:q-Gaussian}
    For $\mu\in \bbR^n$ and $\Sigma \in \calS_{++}^{n}$,
    a $q$-Gaussian $N_{q}(\mu, \Sigma
    )$ is an $n$-dimensional random variable whose density function is given by \begin{align}
      \varphi(x) := \frac{1}{Z_{q}}\exp_{q}\left(-\frac{(x-\mu)^{\top}\Sigma^{-1}(x-\mu)}{(n+4)-(n+2)q}\right)
    \end{align}
    where
    \begin{align*}
      Z_{q}:= \det(\Sigma)^{1/2}\left(\pi \frac{(n+4)-(n+2)q}{1-q}\right)^{n/2}\frac{\Gamma\left(\frac{2-q}{1-q}\right)}{\Gamma\left(\frac{2-q}{1-q} + \frac{n}{2}\right)}. 
    \end{align*}
  \end{definition}
  \begin{proposition}[Statistics of $q$-Gaussian \cite{furuichi_maximum_2009}]
    For any $\mu \in \bbR^n$, and $\Sigma \in \calS_{++}^{n}$, 
    the $q$-Gaussian $N_{q}(\mu, \Sigma)$ satisfies $\mathbb{E}[x] = \mu$ and $\mathbb{V}[x] = \Sigma$.
  \end{proposition}

  \begin{figure}
    \centering
    \includegraphics[width=\hsize]{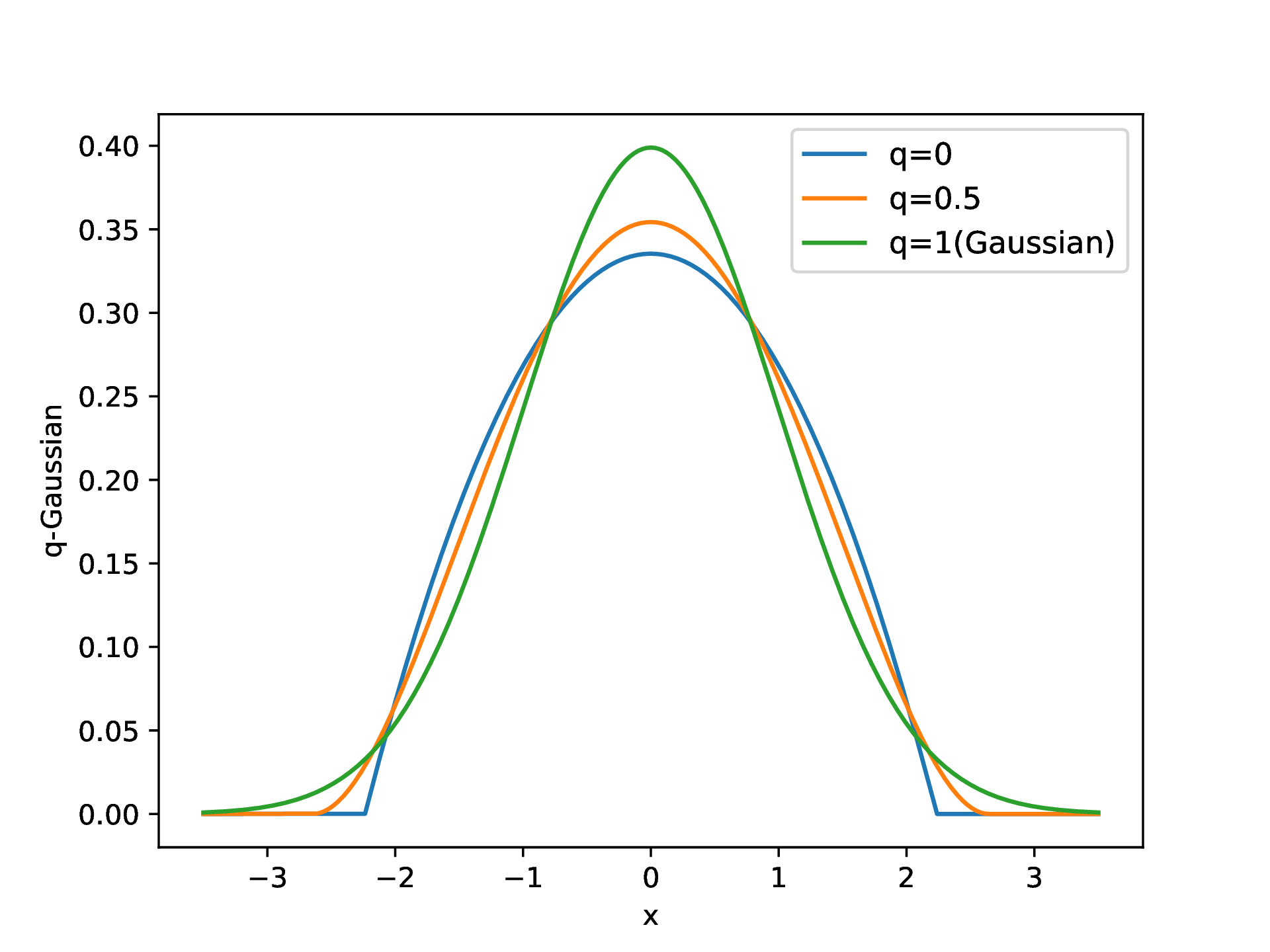}
    \caption{Density functions of the $q$-Gaussian.}
    \label{fig:q-Gaussian}
  \end{figure}

  It follows from \eqref{eq:q-exp} that the support of the $q$-Gaussian $N_q(\mu, \Sigma)$ is bounded, represented as
  \begin{align}
    \begin{split}
        &\supp{\varphi(x)} \\
      &= \left\{x \mid (x-\mu)^{\top}\Sigma^{-1}(x-\mu) < \frac{n+4 - (n+2)q}{1-q}\right\},
    \end{split}\label{eq:Gaussian_support}
  \end{align}
  where $\varphi$ is probability density function of $N_q(\mu, \Sigma)$.
    This implies the closer $q$ is to $0$, the smaller the support becomes. 
    See Fig.~\ref{fig:q-Gaussian} for the density functions of $q$-Gaussian $N_{q}(0, 1)$ as $q$ varies. 
\section{Tsallis Entropy Regularized Optimal Control Problem}
  \label{sec:TROC}

  \subsection{Problem formulation}
  Consider a discrete-time system with the state $x_{k}\in \mathbb{X}\subset \mathbb{R}
  ^{n}$ and the control input $u_{k}\in \mathbb{U}\subset \mathbb{R}^{m}$.
  Conditional state transition probability distribution of $x_{k+1}$ under
  $x_{k}= \bm{x}, u_{k}= \bm{u}$, is denoted as $\varphi_{x_{k+1}}(\bm{x}
  '\mid x_{k}= \bm{x}, u_{k}= \bm{u})$. Also, $\varphi_{x_0}(\bm{x})$ denotes the
  initial distribution. Under this setting, we formulate Tsallis Entropy Regularized
  Optimal Control Problem (TROC).
  \begin{problem}[TROC]\label{prob:TROC} 
  Let the terminal time be $T \in \mathbb{Z}_{>0}$. 
  Find the stochastic state feedback control policy $\pi_{k}(\bm{u}\mid\bm{x}) = \varphi
    _{u \mid x}(\bm{u}\mid x_{k}= \bm{x}), k = 0, \ldots, T-1$, that minimizes the
    cost function $J$ given by
    \begin{align}\label{eq:obj-troc}
      \begin{split}J(&\left\{\pi_k\right\}_{k=0}^{T-1}) := \mathbb{E}\left[l_{T}(x_{T})\right] \\ 
        &+ \sum_{k=0}^{T-1}\mathbb{E}\left[l_{k}(x_{k}, u_{k}) - \lambda\mathcal{H}_{q}(\pi_k(u_{k}\mid x_{k}))\right],
      \end{split}
    \end{align}
    where $\lambda > 0$ and $\mathcal{H}_{q}$ is the conditioned deformed $q$-entropy
    \begin{equation}
      \mathcal{H}_{q}(\pi(u\mid \bm{x})) := - \frac{1}{2-q}\left(\int \pi(u\mid \bm
      {x})\log_{q}\pi(u\mid \bm{x}) du - 1\right).
    \end{equation}
  \end{problem}

  The objective of Problem \ref{prob:TROC} is to balance the traditional cost minimization
  (i.e., the sum of running costs $l_{k}(x_{k}, u_{k})$ and the terminal cost
  $l_{T}(x_{T})$) and the maximization of the deformed $q$-entropy of the control
  policy, which encourages exploration in the control policy. The parameter $\lambda$ controls the trade-off between
  these objectives.

  \subsection{Bellman equation for TROC}\label{subsec:Bellman}
  In this section, we introduce the Bellman equation for TROC in a general setting based on \cite{lee_generalized_2020}.
  The state-value function $V^{*}(i, \bm{x})$ is introduced as
  \begin{align}
  \begin{split}
      &V^{*}(i,\bm{x})\\ 
      &:= \min_{\{\pi_k\}_{k=i}^{T-1}}\mathbb{E}[l_{T}(x_{T})]
  +\sum_{k=i}^{T-1}\mathbb{E}[l_{k}(x_{k}, u_{k}) - \lambda\mathcal{H}_{q}(\pi_k(u_{k}\mid x_{k}))]
  \end{split}\label{eq:V}
  \end{align}
  and state-input value function
    \begin{align}\label{eq:Q}
      Q^*_{k}(\bm{x}, \bm{u}) := l_{k}(\bm{x}, \bm{u})
       + \mathbb{E}[V^*(k+1, x_{k+1})\mid x_{k}=\bm{x},u_{k}=\bm{u}].
    \end{align}
  Then, we obtain the following \cite[Theorem 2]{lee_generalized_2020}:
    \begin{proposition}[Bellman Equation for TROC]
      \label{thm:bellman-TROC} For Problem \ref{prob:TROC}, 
      the state-value function satisfies
      \begin{align}\label{eq:bellman-before}
        V^{*}(k, \bm{x})=\min_{\pi_k}\mathbb{E}[Q_{k}^{*}(\bm{x}, u)] - \lambda\mathcal{H}_{q}(\pi(u\mid \bm{x})), 
      \end{align}
      and the optimal control policy is given by  
          \begin{equation}
        \label{eq:trot-optimal-policy}\varphi_{u\mid x}^{*}(\bm{u}\mid \bm{x}) := \exp
        _{q}\left(-\frac{1}{\lambda}Q^*_{k}( \bm{x}, \bm{u}) + C_{k}(\bm{x})\right)
      \end{equation}
      where $C_{k}(\bm{x})$ determined by normalization constraint
      \begin{equation}\label{eq:normalization}
        \int_{\bbU} \varphi^{*}_{u|x}(u|\bm{x}) du = 1.
      \end{equation}  
      is called $q$-potential.
    \end{proposition}

  In the case of the Shannon entropy regularized optimal control problem \cite{neu_unified_2017}, that is,
  when $q=1$, the Bellman equation \eqref{eq:bellman-before} is replaced by
  \begin{align}
    V(k, \bm{x})= -\lambda \log\int \exp\left(-\frac{Q_{k}(\bm{x},u)}{\lambda}\right) du
  \end{align}
  and $C_{k}(\bm{x})=V(k, \bm{x})/\lambda$, and the optimal control policy is given by the soft-max function
  \begin{equation}
    \varphi_{u\mid x}^{*}(\bm{u}\mid \bm{x}) \propto \exp\left(-\frac{Q_{k}(\bm{x}, \bm{u})}{\lambda}\right),
  \end{equation}
  which is positive for all $\bm{u}$.

  On the contrary, when $q$ is not a special value, the distribution in \eqref{eq:trot-optimal-policy}, 
  which is called ent-max \cite{peters_sparse_2019}, makes the control law sparse. There are two parameters, $\lambda$ 
    and $q$, and the smaller the value of $0 \leq q < 1$, the sparser the control law 
    becomes. Making $\lambda>0$ smaller also leads to sparsity, but at the same time, 
    it reduces the entropy.
    It should be emphasized that the ent-max control law for $q\neq 1$ does not satisfy $\varphi_{u\mid x}^{*}(\bm{u}\mid \bm{x}) \propto \exp_q(-Q_{k}(\bm{x}, \bm{u})/\lambda)$ and that 
    the $q$-potential $C_k(x)$
    cannot be expressed explicitly. Therefore, since the optimal control law \eqref{eq:trot-optimal-policy} cannot be obtained, previous studies 
    have used methods to obtain approximate solutions 
    \cite{chen_effective_2018, lee_generalized_2020}. 

\subsection{Robustness of Tsallis entropy regularized control}
  \label{sec:robust}

  In this section, we discuss the robustness of the optimal control policy obtained 
  by solving the Tsallis entropy regularized control problem. 
  It is known that the maximum entropy optimal control problem is robust to cost function variations 
  in the sense that it is equivalent to the following min-max problem \cite[Theorem 4.1]{eysenbach_maximum_2021}. 

\begin{proposition}\label{thm:robust}
  For cost functions $l_k$, the optimal control policies of the min-max problem and the maximum entropy optimal control problem is the same;
  \begin{align}
    \begin{split}
    &\argmin_{\pi_k}\max_{\left\{\Tilde{l}_k\right\}\in \Tilde{L}(\left\{l_k\right\};\varepsilon)}
    \sum_{k=0}^{T}\mathbb{E}_{\pi_k}\left[\tilde{l}_k(x_k, u_k)\right]\\
    &\quad = \argmin_{\pi_k}\sum_{k=0}^{T}\mathbb{E}_{\pi_k}\left[l_k(x_k, u_k)\right] - \lambda \mathcal{H}(\pi_k(u_k \mid x_k)) 
    \end{split}\label{eq:robust}
  \end{align}
  $\Tilde{L}(\left\{l_k\right\};\varepsilon)$ is defined as follows.
  \footnotesize
  \begin{align}  
    \begin{split}
    &\Tilde{L}(\left\{l_k\right\};\varepsilon)\\ 
    &:= \left\{\left\{\tilde{l}_k\right\} \middle| {}^\forall k \  \mathbb{E}_{\pi_k}\left[\lambda\log\left(
      \frac{1}{N} \sum_{u_k=1}^N \exp\left(-\frac{\Delta l_k(x_k,u_k)}{\lambda}\right)
    \right) \right] \leq \varepsilon \right\}
    \end{split}
  \end{align}  
  \normalsize
  with $\Delta l_k(x_k,u_k) := l_k(x_k,u_k) - \tilde{l}_k(x_k,u_k)$.
\end{proposition}
This theorem shows that solving the maximum entropy optimal control problem is equivalent to solving the min-max problem with the robust set $\Tilde{L}(\left\{l_k\right\};\varepsilon)$.

A similar theorem holds for the Tsallis entropy regularized optimal control problem.
As preparation, we first define the $q$-expectation.
\begin{definition}
  For $0<q<2$, the $q$-expectation $\mathbb{E}_p^q[\cdot]$ for a probability density function $p$ over $\mathbb{R}$ is defined as follows.
  \begin{align}
    \mathbb{E}_p^q[f(x)] 
    := \frac{\int_{-\infty}^{\infty} p(x)^q f(x) dx}{\int_{-\infty}^{\infty} p(x)^q dx}.
  \end{align}
\end{definition}
This can similarly be defined for probability density functions over finite sets.

\begin{rmk}
  The $q$-expectation for a probability density function $p$ is 
  equal to the standard expectation for the probability density function $\frac{p(x)^q}{\int_{-\infty}^{\infty}p(x)^q dx}$.
  This distribution is called the escort distribution of $p$.
  When $q=1$, the escort distribution is equivalent to the original distribution.
  When $0<q<1$, the shape of the escort distribution becomes flatter compared to the original distribution.
  This means that in this case, events with low probabilities are given greater weight than those in the standard expectation.
  Conversely, when $1<q<2$, the weight for events with low probabilities is smaller, and the weight for events with high probabilities is larger. (See Fig. \ref{fig:escort}.)
\end{rmk}

    \begin{figure}
    \centering
    \includegraphics[width=\hsize]{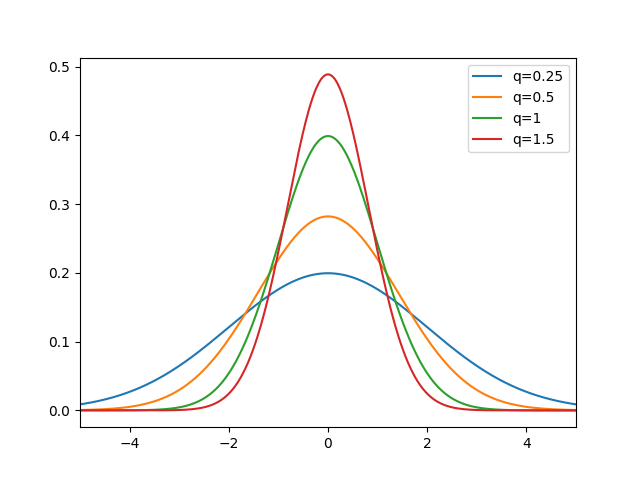}
    \caption{Escort distribution for standard normal distribution}
    \label{fig:escort}
    \end{figure}

 A similar theorem for the Tsallis entropy regularized optimal control problem holds 
 by replacing the expectation in Theorem \ref{thm:robust} 
 with the $q$-expectation and choosing an appropriate robust set.

\begin{theorem}\label{thm:q-robust}
  Let $0<q<1$ and define $\alpha := 2-q$.
  For cost functions $l_k$, the optimal control policies of the min-max problem and the maximum entropy optimal control problem is the same;
  \begin{align}
    &\argmin_{\pi_k}\max_{\left\{\Tilde{l}_k\right\}\in \Tilde{L}_{\alpha}(\left\{l_k\right\};\varepsilon)}
    \sum_{k=0}^{T}\mathbb{E}_{\pi_k}^{\alpha}\left[\tilde{l}_k(x_k, u_k)\right]\\
    &= \argmin_{\pi_k}\sum_{k=0}^{T}\mathbb{E}_{\pi_k}^{\alpha}\left[l_k(x_k, u_k)\right] - \lambda_{\tau, q, \varepsilon} \widehat{D}_q\left(\pi_k(u_k | x_k) \middle\| \psi(u_k | x_k)\right),\label{eq:robust-q}
  \end{align}
  where $\lambda_{\tau, q, \varepsilon} = \tau(1 + (1-q)\varepsilon)$, and  
    \begin{align}
    \begin{split}
      &\Tilde{L}_{\alpha}(\left\{l_k\right\};\varepsilon) := \left\{\left\{\tilde{l}_k\right\} \middle|^{\forall}k \ \ \right.\\
        &\left.\mathbb{E}^{\alpha}_{\pi_k}\left[\tau\log_{\alpha}\left(
        \int \psi(u_k | x_k) \exp_{\alpha}\left(-\frac{\Delta l_k(x_k,u_k)}{\tau}\right)du_k
      \right) \right] \leq \varepsilon\right\}
      \end{split}
      \end{align}
      \begin{align}
      \widehat{D}_q\left(\varphi \middle\| \psi\right) &:=
        \frac{1}{\int \varphi(x)^{2-q} dx} \int \varphi(x) \log_q\left(\frac{\varphi(x)}{\psi(x)}\right) dx \\
        &= \frac{1}{\int \varphi(x)^{2-q} dx} \qKL{\varphi}{\psi},
    \end{align}
    \normalsize
    with $\Delta l_k(x_k,u_k) := l_k(x_k,u_k) - \tilde{l}_k(x_k,u_k)$ and \revise{an arbitrary probability distribution $\psi(u_k | x_k)$,
which specifies $\Tilde{L}_{\alpha}(\left\{l_k\right\};\varepsilon)$.}
  \end{theorem}
  \begin{IEEEproof}
    Consider the following minimization problem.
    \begin{align}\label{prob:min-q}
      \max_{\left\{\Tilde{l}_k\right\}\in \Tilde{L}_{\alpha}(\left\{l_k\right\};\varepsilon)}
      \sum_{k=0}^{T}\mathbb{E}_{\pi_k}^{\alpha}\left[\Delta l_k(x_k, u_k)\right]
    \end{align}
    Since this problem is convex with respect to $\Delta l_k$, 
    if there exists a $\Delta l_k$ that satisfies the following KKT conditions, it becomes the optimal solution.
    
    \begin{align}\label{eq:kkt-q}
    \begin{split}
      &\sum_{k=0}^T \frac{\pi_k(x_k, u_k)^{\alpha}}{\int\int\pi_k(x_k, u_k)^{\alpha} dx_k du_k}\\
      &=\sum_{k=0}^T \frac{\mu_k\pi_k(x_k)^{\alpha}}{\int\pi_k(x_k)^{\alpha} dx_k} \left(\frac{\psi(u_k | x_k)\exp_{\alpha}(-\Delta l_k(x_k,u_k) / \tau)}{\int \psi(u_k | x_k) \exp_{\alpha}(-\Delta l_k(x_k,u_k) / \tau) du_k}\right)^{\alpha}
      \end{split}\\
      & \mathbb{E}^{\alpha}_{\pi_k} \left[\tau \log_{\alpha}\left(\int \psi(u_k | x_k)\exp_{\alpha}\left(-\frac{\Delta l_k(x_k,u_k)}{\tau}\right) du_k \right)\right] = \varepsilon \\
      & \mu_k \geq 0
    \end{align}    
    
    Here, we define $\Delta l_k^*$ and $\mu_k^*$ as follows.
    \begin{align}
      &\Delta l_k^*(x_k, u_k) := -\log_{\alpha}\left(\frac{\pi_k(u_k \mid u_k)}{\psi(u_k \mid x_k)}\exp_{\alpha}\left(\frac{\varepsilon}{\tau}\right)\right) \\
      &\mu_k^* := 1
    \end{align}
    Then $\Delta l_k^*, \mu_k^*$ satisfies KKT conditions \eqref{eq:kkt-q},
    thus these are optimal solution for problem \eqref{prob:min-q}.
    Therefore, the optimal value of problem \eqref{eq:robust-q} is
    \begin{align}
      &\max_{\left\{\Tilde{l}_k\right\}\in \Tilde{L}_{\alpha}(\left\{l_k\right\};\varepsilon)}
        \sum_{k=0}^{T}\mathbb{E}_{\pi_k}^{\alpha}\left[\tilde{l}_k(x_k, u_k)\right] \\
      &=\sum_{k=0}^{T}\mathbb{E}_{\pi_k}^{\alpha}\left[l_k(x_k, u_k)\right] - \sum_{k=0}^{T}\mathbb{E}_{\pi_k}^{\alpha}\left[\Delta l_k^*(x_k, u_k)\right] \\
      &= \sum_{k=0}^{T}\mathbb{E}_{\pi_k}^{\alpha}\left[l_k(x_k, u_k)\right] - \lambda_{\tau, q, \varepsilon} \widehat{D}_q\left(\pi_k(u_k | x_k) \middle\| \psi(u_k | x_k)\right) + \textrm{const.}
    \end{align}
  \end{IEEEproof}
  $\widehat{D}_q$ is called the normalized $q$-KL divergence and it becomes equivalent to Shannon KL divergence when $q \rightarrow 1$.
  When $q \rightarrow 1$ in Theorem \ref{thm:q-robust}, it coincides with Proposition \ref{thm:robust}.
  As shown in \eqref{eq:robust}, the solution to the min-max problem in Proposition \ref{thm:robust} is independent of $\varepsilon$. However, in its generalization, \eqref{eq:robust-q},
the coefficient of the normalized $q$-KL divergence term decreases as $\varepsilon$ increases when $q<1$.

  \begin{rmk}
  The differences between the objective functions \eqref{eq:obj-troc} and the right hand side of \eqref{eq:robust-q} lie in how the expectation of the cost function is taken and in their entropy terms.
The expectation of the cost function is an $\alpha$-expectation, or equivalently a $(2-q)$-expectation, which means it weights events with higher probabilities more heavily.
The entropy term is a $q$-entropy normalized by $\int \pi_k(x_k,u_k)^\alpha dx_k du_k$.
  \end{rmk}


\section{$q$-Linear Quadratic Regulator}
  \label{sec:qLQR}

  \subsection{Riccati equations}

  The Linear Quadratic Regulator (LQR) is an optimal control problem for linear
  systems with quadratic cost functions, which can be solved analytically. In
  this section, we formulate an optimal control problem that incorporates Tsallis
  entropy as a regularization term for LQR ($q$-LQR) and discuss its properties; See \cite[Proposition 1]{ito_maximum_2024} and \cite[Proposition 1]{ito_maximum_2023} for the Shannon entropy case.

  We consider the state $x_{k}\in \mathbb{R}^{n}$ and
  control input $u_{k}\in \mathbb{R}^{m}$ to follow the linear system given by
  \begin{align}
    x_{k+1} & = A_{k}x_{k}+ B_{k}u_{k}. 
  \end{align}
  The stage cost at each time is given by the following quadratic form:
  \begin{align}
      l_{k}(x, u) & := x^{\top}Q_{k}x + 2x^{\top}S_{k}u + u^{\top}R_{k}u \quad (k = 0,...,T-1)\notag \\
      l_{T}(x)    & := x^{\top}Q_{T}x,\label{eq:lqr-cost}
  \end{align}
    where $Q_{k}\in \calS_{++}^{n}, R_{k}\in \calS_{++}^{m}$ and $S_k\in \bbR^{n\times m}$.
    While we need to solve \eqref{eq:trot-optimal-policy} and \eqref{eq:normalization} to obtain the $q$-potential for $q$-KL control, that for $q$-LQR can be expressed explicitly. Furthermore, as shown in the following 
    theorem, the optimal control law is represented by a $q$-Gaussian. This 
    corresponds to the fact that the optimal control law in Shannon entropic LQR is represented by a Gaussian.
  
  \begin{theorem}
    \label{thm:lqr-TROC}
  For Problem \ref{prob:TROC} with quadratic cost functions \eqref{eq:lqr-cost}, the value function is represented as $V^*(k, x) = x^{\top}
  \Pi_{k}x + \textrm{const.}$, and the optimal control policy is $q$-Gaussian with mean $\mu$ and variance $\Sigma$ as follows:
  \begin{align}
    &\mu_k := K_{k} x_k, \label{eq:lqr-mu}\\
    &\Sigma_k^{-1} := \frac{(n+4)-(n+2)q}{\lambda}\eta\widetilde{R}_{k},\label{eq:lqr-sigma} \\
    &\eta_k := \left\{
      \det(\widetilde{R}_k)^{-1/2} \left(\frac{\pi \lambda}{1-q}\right)^{n/2} 
      \frac{\Gamma\left(\frac{2-q}{1-q}\right)}{\Gamma\left(\frac{2-q}{1-q} + \frac{n}{2}\right)} 
    \right\}^{\frac{2(1-q)}{(n+2)-nq}},
  \end{align}
  wtih
  \begin{align}
    \widetilde{R}_{k} & := R_{k}+ B_{k}^{\top}\Pi_{k+1}B_{k}                                                  \\
    \widetilde{S}_{k} & := S_{k}+ B_{k}^{\top}\Pi_{k+1}A_{k}                                                  \\
    \widetilde{Q}_{k} & := Q_{k}+ A_{k}^{\top}\Pi_{k+1}A_{k}                                                  \\
    \Pi_{k}           & := \widetilde{Q}_{k}- \widetilde{S}_{k}\widetilde{R}_{k}^{-1}\widetilde{S}_{k}^{\top}\\
    K_{k} &:= -\widetilde{R}_{k}^{-1}\widetilde{S}_{k}^{\top}. 
  \end{align}
  
   Furthermore, $q$-potential $C_k(x)$ in \eqref{eq:trot-optimal-policy} can be expressed as
    \begin{equation}
      C_k(x) = \log_q(\eta_k^{-\frac{1}{1-q}}).
    \end{equation}

  \end{theorem}
  \begin{IEEEproof}
    From Theorem \ref{thm:bellman-TROC}, we have
    $V(T,\bm{x}) = \bm{x}^{\top}Q_{T}\bm{x}$. Assuming that $V(k+1, \bm{x}) = \bm
    {x}^{\top}\Pi_{k+1}\bm{x}+ \textrm{const.}$, the Bellman equation becomes below:
    \begin{align}
      V(k, \bm{x}) & = \min_{\varphi_{u\mid x}}\mathbb{E}\left[Q_{k}(\bm{x}, u_{k}) - \lambda \mathcal{H}_{q}(\varphi_{u\mid x}(\cdot\mid \bm{x}))\right] \\
      \begin{split}
        &=\min_{\varphi_{u\mid x}}\mathbb{E}[\bm{x}^{\top}Q_{k}\bm{x}+ 2\bm{x}^{\top}S_{k}u_{k}+ u_{k}^{\top}R_{k}u_{k}\\ 
        &\hspace{40pt} + (A_{k}\bm{x}+ B_{k}u_{k})^{\top}\Pi_{k+1}(A_{k}\bm{x}+ B_{k}u_{k})\\
        &\hspace{40pt} - \lambda\mathcal{H}_{q}(\varphi_{u\mid x}(\cdot\mid \bm{x}))] + \textrm{const.}        
      \end{split}      \\
      \begin{split}
          & = \bm{x}^{\top}\Pi_{k}\bm{x} + \min_{\varphi_{u\mid x}}\mathbb{E}[ (u_{k}-K_{k}\bm{x})^{\top}\widetilde{R}_{k}(u_{k}-K_{k}\bm{x}) \\
          &\hspace{80pt} - \lambda\mathcal{H}_{q}(\varphi_{u\mid x}(\cdot\mid \bm{x})) ] + \textrm{const.}
      \end{split}\label{eq:lqr-min}
    \end{align}
    By Lemma \ref{lem:TRQ-quadratic} in the Appendix, the minimizer of \eqref{eq:lqr-min} is the density
    function of a $q$-Gaussian with mean $\mu$ in \eqref{eq:lqr-mu} and variance $\Sigma$ in \eqref{eq:lqr-sigma}. 
    Since the minimum value does not depend on $\bm{x}$, it holds that
    $V(k, \bm{x}) = \bm{x}^{\top}\Pi_{k}\bm{x}+ \textrm{const.}$. Thus, the obtained result is proven inductively.
  \end{IEEEproof}

  This theorem shows that the linear feedback with the same gain matrix as LQR with additive noise $w_{k}\sim N_q(0,\Sigma_k)$ is optimal.
  The resulting closed-loop dynamics is
  \begin{align}\label{eq:closed-loop}
    x_{k+1} & = (A_{k} + B_{k} K_{k})x_{k} + B_{k} w_{k}.
  \end{align}
  Since the $q$-Gaussian has bounded support, the support for the state will also be bounded at any time if the initial state distribution is bounded.
    Furthermore, the region of the support can be easily estimated by \eqref{eq:Gaussian_support} and \eqref{eq:closed-loop}.  
      This can be particularly important in real-world applications such as robotics, 
  where the system can fail if inputs or states move outside their stable operating region.

    \begin{rmk}
      Unlike regular Gaussian, 
      $q$-Gaussian generally does not possess the reproductive property.
      This means the state variable does not obey $q$-Gaussian. 
    \end{rmk}

  \subsection{Numerical example}
  In this section, we solve Problem \ref{prob:TROC} for $q=0.5$, $\lambda=0.1$ , $T=20$ with following system and cost function
  \begin{align}
      A = \begin{bmatrix}
          -3/2 & 1/2\\
          0 & 1/2
      \end{bmatrix},\  
      B= \begin{bmatrix}
          1 \\ 0
      \end{bmatrix},\  
      Q = \begin{bmatrix}
          1 & 0\\
          0 & 1
      \end{bmatrix},\  
      S= \begin{bmatrix}
          0\\
          0
      \end{bmatrix},\  
      R= 1.
  \end{align}

\begin{figure}[ht]        
    \begin{minipage}[b]{\linewidth}
        \includegraphics[width=\linewidth]{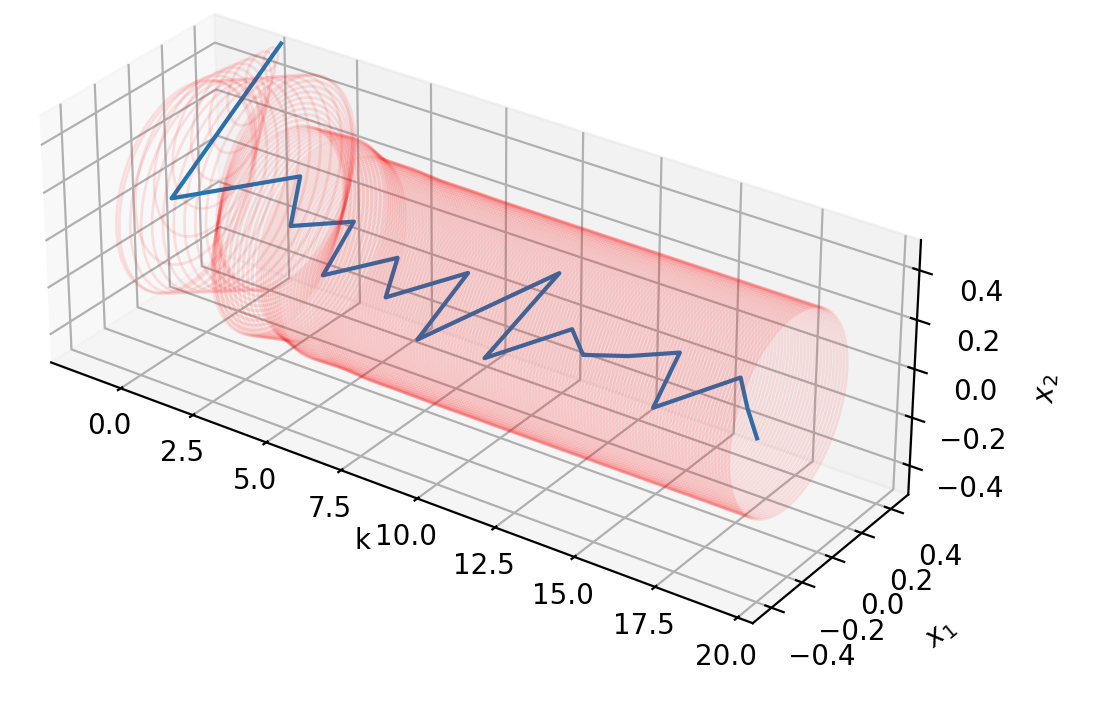}
        \subcaption{Optimal state trajectory for $q$-LQR, $q=0.5$}
        \label{fig:q-LQR}
        \vfill
        \includegraphics[width=\linewidth]{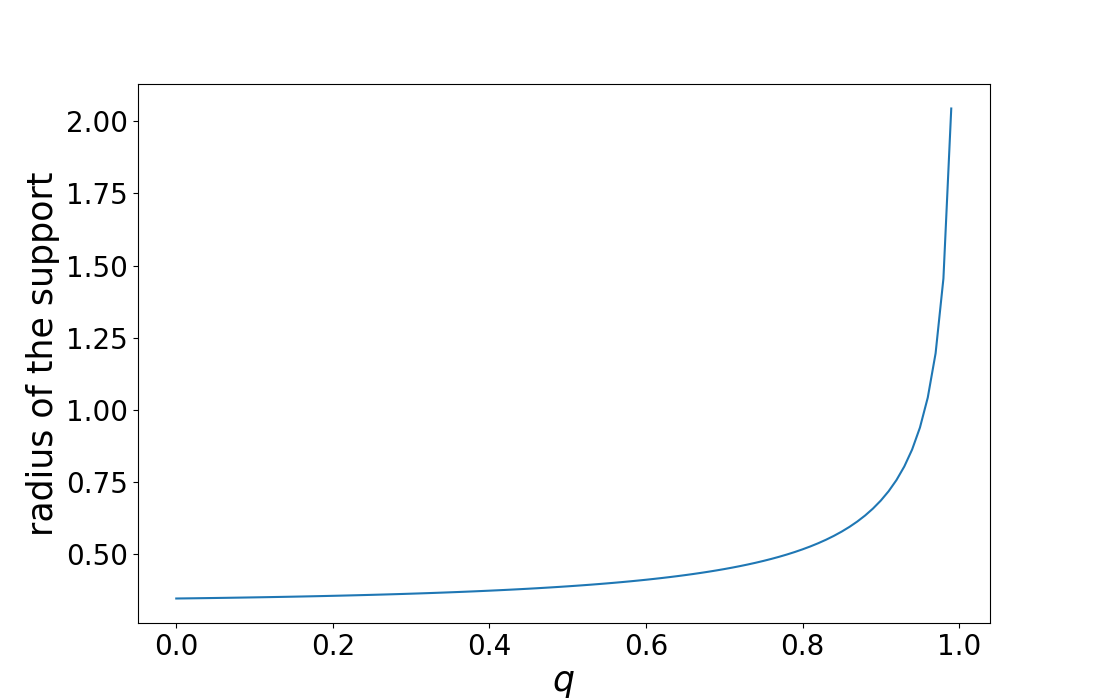}
        \subcaption{Area of the support of the noise $w_T$}
        \label{fig:q-LQR-boundary}
    \end{minipage}
\caption{
Simulation results of $q$-LQR contol. 
(a) State trajectory using $q$-LQR. The blue line represents the actual trajectory. All possible states are contained within the red region.
(b) The radius of the support of the state at $k=20$ for each $q$. 
}\label{fig:q-LQR-all}
\end{figure}

In Fig.~\ref{fig:q-LQR}, the blue line represents a sample path of the optimally controlled state, 
    and the red region contains the support of the state distribution (See Appendix \ref{seq:qLQR-appendix} for details).
    In the case of Shannon entropy ($q=1$), the support of the state distribution is unbounded, and no such boundary exists. 
    As a result, unexpected outliers may occur, 
    and the controlled system may exit a safe region in the real world.
    Such risks can be avoided by using Tsallis entropy regularization.

Figure \ref{fig:q-LQR-boundary} shows the relationship between values of $q$ and corresponding radius of the support at $k=20$, indicating that the larger the $q$, the larger the support becomes.

 \section{$q$-Kullback-Leibler control}
  \label{sec:qKL}

  \subsection{Kullback-Leibler control}
  Kullback-Leibler (KL) control is a control problem having costs described in terms of the KL divergence, enabling efficient numerical solutions to nonlinear optimal
  control problems \cite{ito_kullbackleibler_2022}. Since KL control can also be interpreted as an entropy-regularized
  optimal control problem, it can be extended to the Tsallis entropy framework by
  replacing KL divergence with $q$-KL divergence.
    The advantage over MDP setting is that there is no need to estimate the input-value function.

  In this section, we assume states are defined on a finite set $\mathbb{X}=\{1 ,
  ...,n\}$, and control inputs are given by a transition matrix
  $P\in\mathbb{R}^{n\times n}$. That is, if at time $k$, the state is distributed
  according to the probability vector $\varphi_{k}\in \mathbb{R}^{n}$, then at time $k+1$,
  the state distributes according to $\varphi_{k+1}=P^{\pi}\varphi_{k}$ with input $P$.
  Under this setting, we formulate the $q$-KL control problem as follows:
  \begin{problem}
    [$q$-KL Control Problem]\label{prob:qKL} Consider a Markov process
    $\varphi_{k}^{\pi}$ with transition matrix $P_{k}^{\pi}$. For the initial distribution
    $\varphi_{0}$, state stage cost $l\in\mathbb{R}^{n}$, transition matrix $P^{0}$,
    terminal time $T\in\mathbb{Z}_{>0}$, and $\lambda > 0$, find the transition
    matrices $P_{k}^{\pi}, k=0,...,T-1$ that minimize the cost function $J$
    given by
    \begin{align}\label{eq:qKL-cost}
      J(\pi):=l^{\top}\varphi_{T}^{\pi}+ \sum_{k=0}^{T-1}\left(l^{\top}\varphi_{k}^{\pi} + \lambda \qKL{P_{k}^{\pi} \varphi_{k}^{\pi}}{P^{0} \varphi_{k}^\pi}\right).
    \end{align}
    Here, $\qKL{\varphi}{\psi}$ is the $q$-KL divergence in the discrete case, 
    \begin{align}
      \qKL{\varphi}{\psi}:= \frac{1}{2-q}\left(\sum_{i}\varphi_{i}\log_{q}\frac{\varphi_{i}}{\psi_{i}} - 1\right).
    \end{align}
  \end{problem}
  According to the objective function \eqref{eq:qKL-cost}, 
  the goal of Problem \ref{prob:qKL} is to minimize the cost associated with the state at each time point, 
  while also minimizing the $q$-KL divergence between the state transition matrix $P_{k}^{\pi}$ 
  conditioned on the state and the given transition matrix $P^{0}$. 
  Since $P_0$ represents the transition probabilities in the absence of control, 
  the cost of changing the transition matrix from $P_0$ to $P_{k}^{\pi}$ is expressed using $q$-KL divergence. 
  In particular, similar to the conventional KL divergence, 
  \begin{align}\label{eq:KLsupport}
      \supp{\varphi} \subset \supp{\psi}
  \end{align}
  is needed to make $\qKL{\varphi}{\psi}$ finite. 
  This implies that only transitions that can occur without control can be realized.
  
Similar to KL control, there is no need to estimate the
input-value function for the $q$-KL control setting
  
  \begin{theorem}\label{thm:qKL}
    For Problem \ref{prob:qKL}, the optimal control policy $P_{k}^{*}$ is given by 
    \begin{equation}\label{eq:KLopt}
      (P_{k}^{*})_{ij}:=P^{0}_{ij}\exp_{q}\left(-\frac{1}{\lambda}V^*(k+1)_{i} + C_{k}(j)\right)
    \end{equation} 
    where $C_{k}(j)$ and $V^{*}$ are determined by $\sum_{i}(P_{k}^{*})_{ij} = 1$ and 
    \begin{align}
    V^{*}(T)_{j} &= l_{j}\\
    \begin{split}
    V^{*}(k)_{j} &= l_{j}
        +\left\{\lambda\qKL{(P_k^*)_{:j}}{(P^0)_{:j}}+ V^{*}(k+1)^{\top}(P_{k}^{*})_{:j}\right\} \\
        &\hspace{140pt} (k = 0,...,T-1),    
    \end{split}\label{eq:bellman-qKL}
    \end{align} 
    where $(P)_{:j}$ denotes the $j$-th column of $P$.
  \end{theorem} 
  \begin{IEEEproof}
We show $V^{*}$ is the optimal state-value function.     Similar to \eqref{eq:bellman-before}, let us consider
\begin{align}
        V(k)_{j}= l_{j}
        +\min_\pi \left\{\lambda\qKL{(P_k^\pi)_{:j}}{(P^0)_{:j}}+ V(k+1)^{\top}(P_{k}^{\pi})_{:j}\right\}.        
\end{align}
    Following the same reasoning as the proof of Lemma \ref{lem:TRQ}, we can find the minimizer. 
    The KKT conditions are given by
    \begin{align}
      V(k+1)_{i} + \lambda\log_{q}\frac{(P_{k}^{\pi})_{ij}}{P^{0}_{ij}} - C_{k}(j)' = 0,\label{eq:qKL-KKT-1} \\
      C_{j}(k)'\left(\sum_{i}(P_{k}^{\pi})_{ij} - 1\right) = 0 \label{eq:qKL-KKT-2}.
    \end{align}
    From \eqref{eq:qKL-KKT-1}, the minimizer is given by $P_{k}^{*}$ in \eqref{eq:KLopt}.
  \end{IEEEproof}

  \begin{rmk}
    One of the advantages of KL control with $q=1$ is linear solvability 
    \cite[Section 2]{todorov_linearly-solvable_2006}. However, the Bellman equation for $q$-KL with $q \neq 1$ does not reduce to a linear equation because the $q$-exponential function 
    does not satisfy the exponential law.
  \end{rmk}


\revise{
\subsection{Efficient calculation based on parallelization and network sparsity}

For the Shannon-entropy regularized KL-control case, i.e., $q=1$, \eqref{eq:KLopt} reduces to 
\begin{equation}
(P_k^*)_{ij} \propto P^0_{ij}\exp\left(-\frac{1}{\lambda}V^*(k+1)_i\right) 
\end{equation}
for each $(j, k)$. 
In contrast, no closed-form expression is available for the Tsallis-entropy case (\(q \neq 1\)).
Instead, we propose a scalable numerical procedure.

To compute the optimal control policy in Theorem \ref{thm:qKL}, for $k=T-1,\ T-2,\ \ldots, 0$, we need to 
\begin{itemize}
    \item solve $F_{jk}(C_k(j)) = 0$, where 
    \begin{equation}
    F_{jk}(C) := \sum_{i} P^0_{ij} \,\exp_q \left( - \frac{1}{\lambda}V^*(k+1)_i + C \right) - 1,
    \end{equation}
    \item obtain $(V^*(k))_j$ by \eqref{eq:bellman-qKL},
\end{itemize}
for $j=1,\ldots,n$. 
It should be emphasized that both of them are closed with respect to $j$. This separable structure allows for straightforward parallel computation. For instance, one can employ the Newton iteration
\begin{equation}\label{eq4}
        C_{k}(j)^{(t+1)} = C_{k}(j)^{(t)} - \beta \frac{F_{jk}(C_{k}(j)^{(t)})}{F_{jk}'(C_{k}(j)^{(t)})}
    \end{equation}
    with damping factor $\beta \in (0,1)$ to enhance global convergence. 

Moreover, in many practical applications the underlying network is sparse, in the sense that many entries of \(P^0\) are zero.
Such sparsity can be exploited to reduce memory requirements and to accelerate computations, for example by adopting the sparse matrix data structures provided in SciPy~\cite{virtanen_scipy_2020}.
Furthermore, since the system dynamics \(P^0\) are time-invariant (i.e., independent of \(k\)), the sparsity structure needs to be constructed only once and reused for all time steps.

}

\revise{
  \subsection{Numerical example}\label{subsec:qKL-ex}
    In this section, we formulate the transportation problem between supply and demand nodes
    on a road network as a $q$-KL control problem and demonstrate its properties.
    All computations were performed on a system equipped with an 11th Gen
    Intel\textsuperscript{\textregistered} Core\texttrademark{} i5-1135G7 CPU @ 2.40\,GHz
    and an Intel\textsuperscript{\textregistered} Iris\texttrademark{} Xe Graphics GPU.
    As the dataset, we used the road network data for the central area of Kyoto City,
    obtained from OpenStreetMap (© OpenStreetMap contributors) via
    the OSMnx library~\cite{boeing_osmnx_2025} in Python.
    The network consists of 10,362 nodes and 29,840 edges.
    The numbers of supply and demand nodes were set to 200 and 300, respectively.
    The supply nodes were uniformly distributed across the network,
    while the demand nodes were located in the lower central area of the network,
    corresponding to the location of Kyoto Station.
    We designated the supply and demand nodes as the origin and destination of the travel agents,
    respectively, and all the other nodes as transit nodes, as shown in Fig.~\ref{fig:kyoto}.
    In other words, we consider a scenario representing travel demand
    from all areas of Kyoto City to Kyoto Station.
    This setting can be formulated as a $q$-KL control problem.
    Let the set of nodes be $N = \{1, 2, \dots, n\}$,
    the set of destination nodes be $D$, and the set of origin nodes be $S$. For each $j \in N$, 
    \begin{align}
        l_{j} &:= \begin{cases}
            0 & \textrm{if $j \in D$} \\
            10 c_j & \textrm{otherwise}
        \end{cases}\\
        P^{0}_{ij} &\propto \begin{cases}
            \exp(-d_{ij}) & \textrm{if node $i$ and $j$ are connected}  \\
            0 & \textrm{otherwise}
        \end{cases}\\
        \varphi_0(j) &:= \begin{cases}
            \frac{1}{|S|} & \textrm{if $i \in S$} \\
            0 & \textrm{otherwise}
        \end{cases}
    \end{align}
    where 
    $c_j$[km] (resp. $d_{ij}$[km]) denotes the Euclidean distance from $j$ to the Kyoto Station (resp. $j$ to $i$)
    on the Mercator projection map.
    By formulating the problem in this way, we can expect the optimal trajectories to flow from
    the initial distribution $\varphi_0$ to $D$ along paths of minimal total distance.
        

    Let $T = 50$, $\lambda =1$, and compute the transport trajectories for $q = 1.0, 0.9, 0.5$ and $0.1$.
    The corresponding computation times for the standard and accelerated Newton methods
    are summarized in Table~\ref{tab:comp_time}.
    We observe that the accelerated version is nearly $10^4$ times faster
    than the naively implemented Newton method.
    We further note that, in this problem setting, the memory usage of matrix $P$
    is dramatically reduced from about 800 MB to 0.5 MB by using sparse matrix data structures,
    which also enhance the speed of the computations.
    The trajectories for $q = 1$ and $q = 0.5$ are shown in
    Figs.~\ref{fig:kl-result} and \ref{fig:qkl-result}, respectively.
    For $q = 1$, the resulting trajectories cover all edges in the network.
    By contrast, when $q = 0.5$, the trajectories are significantly sparser,
    with only about $17.4\%$ of the edges being utilized
    ($6.5\%$ for the unregularized problem, i.e., $\lambda=0$).

    \begin{table}[h]
    \centering
    \caption{\revise{Computation times for various $q$ values}}
    \label{tab:comp_time}
    \revise{
        \begin{tabular}{ccccc}
        \toprule
        $q$ & $1.0$ & $0.9$ & $0.5$ & $0.1$ \\
        \midrule
        Standard Newton method [sec] & $574$ & $2604$ & $1493$ & $2189$ \\
        Accelerated Newton method [sec] & $0.10$ & $0.24$ & $0.13$ & $0.29$ \\
        \bottomrule
        \end{tabular}
    }
    \end{table}

    \begin{figure*}[ht]
        \centering
        \begin{minipage}{0.32\linewidth}
            \centering
            \includegraphics[width=\linewidth,height=0.995\linewidth,keepaspectratio=false]{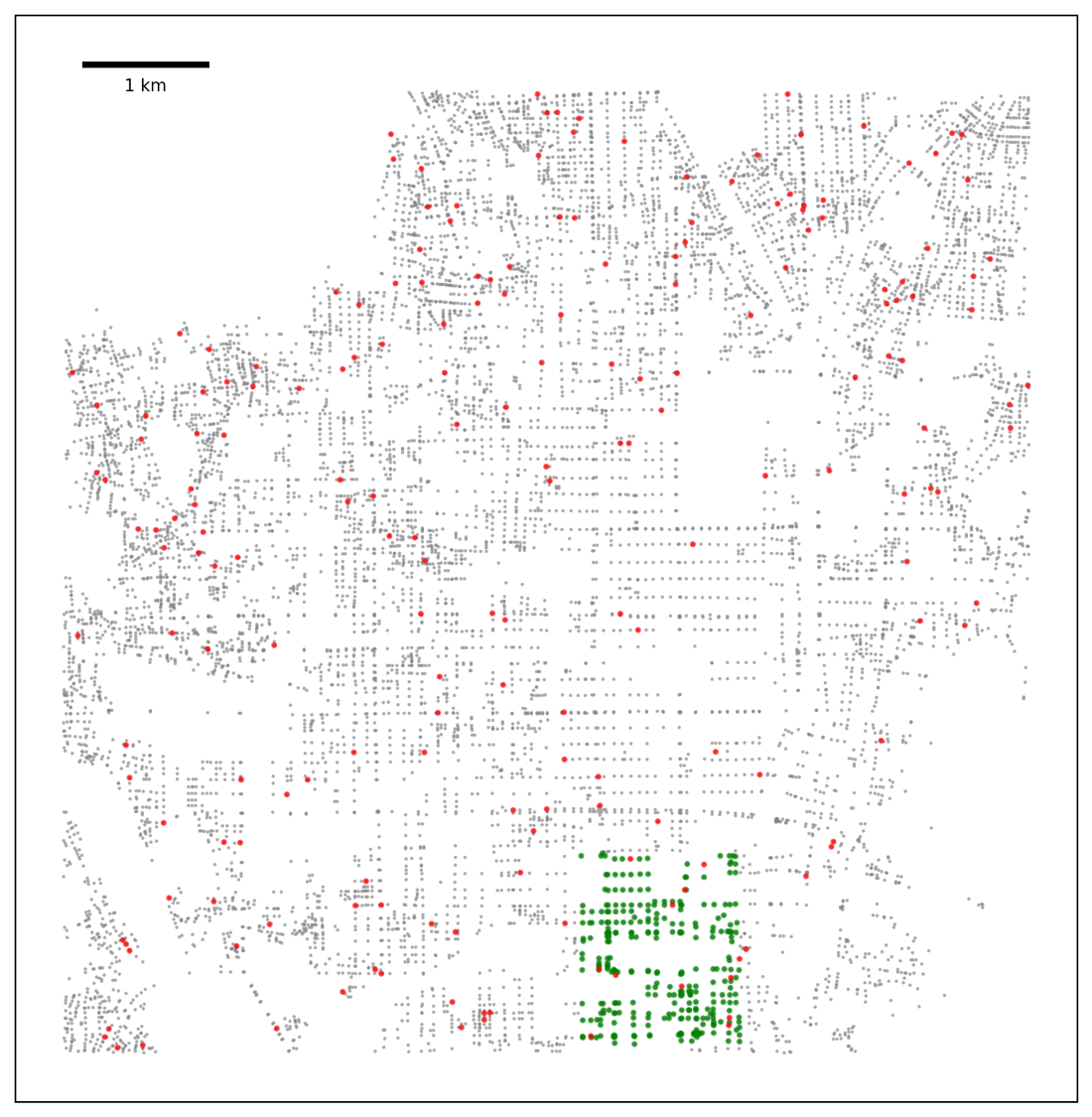} 
            \subcaption{\revise{Nodes of Kyoto road network}}
            \label{fig:kyoto}
        \end{minipage}
        \begin{minipage}{0.32\linewidth}
            \centering
            \includegraphics[width=\linewidth]{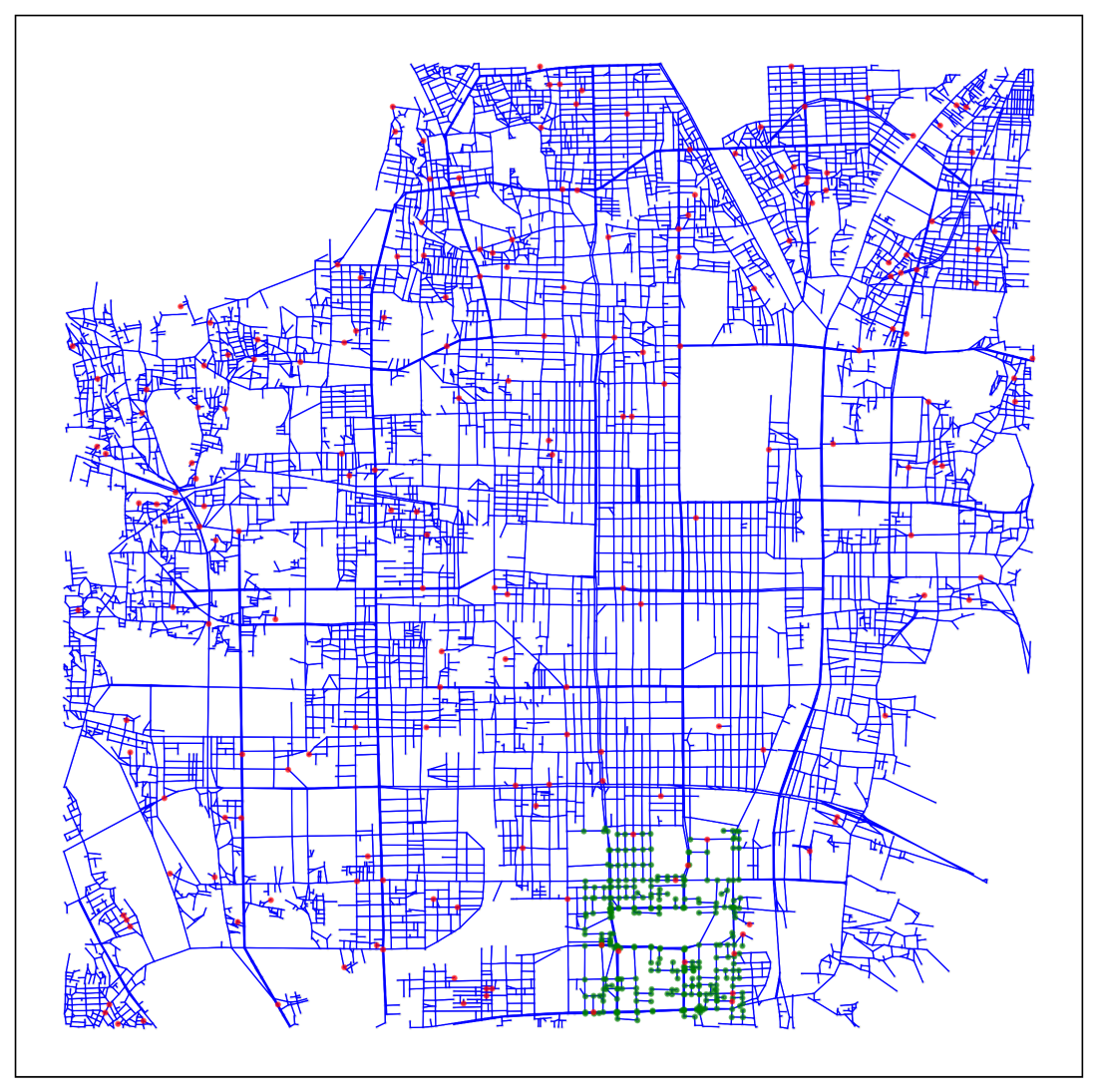}
            \subcaption{\revise{Trajectories for KL control ($q=1$)}}
            \label{fig:kl-result}
        \end{minipage}
        \begin{minipage}{0.32\linewidth}
            \centering
            \includegraphics[width=\linewidth]{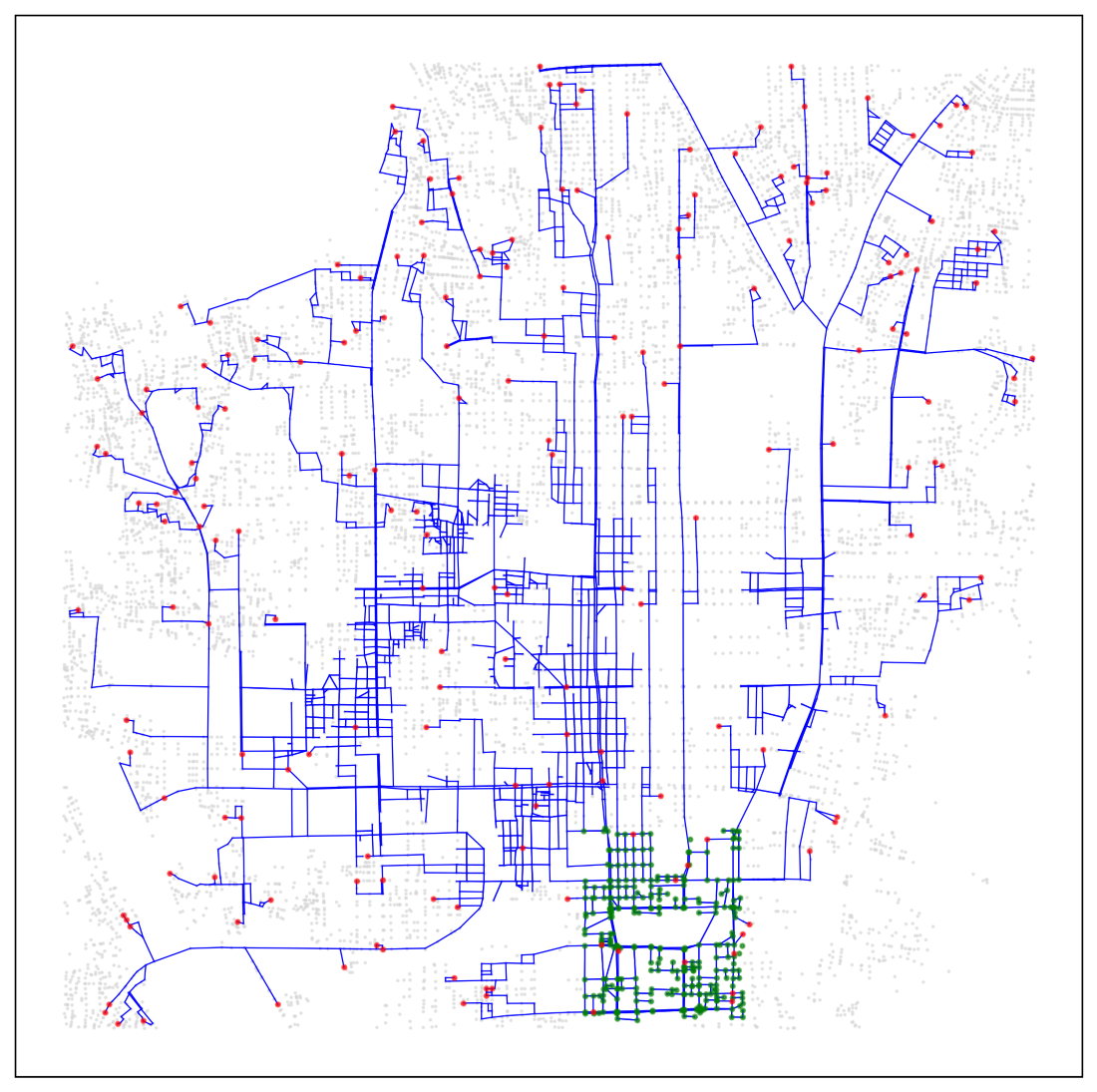}
            \subcaption{\revise{Trajectories for $q$-KL control ($q=0.5$)}}
            \label{fig:qkl-result}
        \end{minipage}
        \caption{
        \revise{
            Simulation results for the Kyoto road network. 
            (a) Node locations: red nodes denote the supply nodes (origins of travel agents),
            green nodes indicate the demand nodes (destinations),
            and gray nodes represent transit nodes. 
            (b) and (c) show the transport trajectories for $q=1$ and $q=0.5$, respectively.
            The $q$-KL control with $q=0.5$ produces sparse trajectories. 
        }
        }
    \end{figure*}

}



 \section{Conclusion}
  \label{sec:conclusion}
    In this study, we formulated the KL control problem and linear quadratic regulator 
    problem with Tsallis entropy regularization term for the case of entropic index 
    $0 \leq q < 1$, and proposed a method to obtain the optimal control law without 
    approximation.
  
  Through numerical experiments, we demonstrated the utility of this approach for obtaining solutions 
  that are both high in entropy and sparse.
  We extended existing results regarding the robustness of problems similar to TROC and demonstrated the dependence on $\varepsilon$ that occurs only when $q\neq1$.
  Robustness of TROC is currently under investigation.

\revise{This work represents a first step toward constructing a Tsallis-entropy-based framework, and many aspects remain to be investigated. For example, the rigorous connection between Problem 1 and Theorem 1 remains open. 
As another example, since Tsallis entropy has been used to regularize optimal transport problems in \cite{bao_sparse_2022}, its control-theoretic counterpart, namely, density control \cite{Ito2023, ito_maximum_2024, ito_maximum_2023}, appears to be a promising research direction (the Sinkhorn iteration cannot be directly applied). 
}

\ifCLASSOPTIONcaptionsoff
  \newpage
\fi


\bibliographystyle{IEEEtran}
\bibliography{rev_main}

\appendices

  \section{Tsallis entropy regularized optimization}

  We prove two lemmas on optimization problems with a Tsallis entropy regularization term.
  \begin{lemma}\label{lem:TRQ}
    For a real-valued function $Q$ and $\lambda > 0$, 
            \begin{align}\label{eq:opt_p}
      \varphi^{*}(u) := \exp_q\left(-\frac{1}{\lambda}Q(u) + C\right)
    \end{align}
    where $C$ is a constant determined by $\int \varphi^{*}(u) du = 1$
    is a unique minimizer of 
    \begin{align}\label{eq:TRQ-objective}
      J(\varphi) := \mathbb{E}_{u\sim \varphi}\left[Q(u)\right] - \lambda\mathcal{H}_{q}(\varphi). 
    \end{align}
  \end{lemma}
  \begin{IEEEproof}
    We derive the optimal solution without the non-negativity constraint $\varphi(u)\geq 0$
    and show that the solution satisfies the non-negativity property.
    Euler-Lagrange equation for the optimization problem is
    \begin{equation}
      Q(u) - \lambda\log_{q}(\varphi(u)) - C' = 0,\label{eq:TRQ-1}
    \end{equation}
    where $C'$ is Lagrangian multiplier.
    From \eqref{eq:TRQ-1} and \eqref{eq:inv-1}, we obtain \eqref{eq:opt_p}.
    Then, $\varphi^*$ is the unique minimizer of \eqref{eq:TRQ-objective} by the concavity of Tsallis entropy $\mathcal{H}_q$, and it is non-negative. 
  \end{IEEEproof}
  
  In particular, when $Q(u)$ is a quadratic form, the optimal solution becomes a $q$-Gaussian.
  \begin{lemma}\label{lem:TRQ-quadratic}
    For a positive-definite matrix $R \in \mathbb{R}^{n \times n}$ and $\mu \in \mathbb{R}^{n}$, define
    \begin{align}
      Q(u) := (u-\mu)^{\top}R(u-\mu).
    \end{align}
    Then, $\varphi^*$ in \eqref{eq:opt_p} is given by $q$-Gaussian $N_{q}(\mu, \Sigma)$ with
    \begin{align}
      & \Sigma^{-1}:= \frac{(n+4)-(n+2)q}{\lambda}\eta R,                                                                                                                      \\
      & \eta := \left\{ \det(R)^{-1/2} \left(\frac{\pi \lambda}{1-q}\right)^{n/2} \frac{\Gamma\left(\frac{2-q}{1-q}\right)}{\Gamma\left(\frac{2-q}{1-q} + \frac{n}{2}\right)} \right\}^{\frac{2(1-q)}{(n+2)-nq}}.
      \label{eq:eta}
    \end{align}
  \end{lemma}
  \begin{IEEEproof}
    Since Tsallis entropy is translation invariant, we can assume $\mu=0$ without loss of generality. Then, 
    \begin{align}
      \varphi^{*}(u) &= \exp_q\left(-\frac{1}{\lambda}u^{\top}Ru + C\right)\\
      &=\exp_q(C)\exp_q\left(-\frac{1}{\lambda\exp_q(C)^{1-q}} u^{\top}Ru\right). 
    \end{align}
    Thus, $\varphi^{*}$ is $N_q(0,\Sigma)$ with 
    \begin{equation}
      \Sigma^{-1}:= \frac{(n+4)-(n+2)q}{\lambda\exp_q(C)^{1-q}}R. 
    \end{equation}
    From the normalization condition,  
    $\exp_q(C)$ satisfies
    \begin{align*}
      \begin{split}
        &\exp_q(C)^{-1} \\
        &= \det(R)^{-1/2}\left(\frac{\pi \lambda \exp_q(C)^{1-q}}{1-q}\right)^{n/2}\frac{\Gamma\left(\frac{2-q}{1-q}\right)}{\Gamma\left(\frac{2-q}{1-q} + \frac{n}{2}\right)}.  
      \end{split}      
    \end{align*}
    Finally, $\eta := \exp_q(C)^{-(1-q)}$ is equivalent to \eqref{eq:eta}.
  \end{IEEEproof}

\section{Upper bound for support of $q$-LQR optimal state trajectory}\label{seq:qLQR-appendix}
    Here, we derive the region that contains the optimal state trajectory of $q$-LQR given an initial state $x_0$.
    From Theorem \ref{thm:lqr-TROC}, the optimal state trajectory follows the dynamics:
    \begin{align}
        x_{k+1} &= (A_{k} + B_{k} K_{k})x_{k} + B_{k} w_{k}\\
        w_{k} &\sim N_q(0, \Sigma_k), 
    \end{align}
    Therefore, the expected value of the state follows the dynamics $\mathbb{E}[x_{k+1}]=(A_{k} + B_{k} K_{k})\mathbb{E}[x_{k}]$. Hereafter, we denote the expected value of the state as $\Bar{x}_k := \mathbb{E}[x_k]$.

    Under these conditions, the following proposition holds:
    \begin{proposition}
        Let $\left\{x_k\right\}_{k=1}^T$ be the optimal state trajectory of $q$-LQR, and let the initial state $x_0$ be given. Then the support of the state distribution can be expressed as:
        \begin{align}\label{eq:qLQR-supp}
            \supp{\varphi_{x_k}}
            \subset \left\{x_k\middle|\ \|x_k-\Bar{x}_k\|^2 \leq r_k^2 \right\},
        \end{align}
        where $r_k>0$.
    \end{proposition}
    \begin{IEEEproof}
        Hereafter, let $\lambda_{\textup{max}}(A)$ and $\lambda_{\textup{min}}(A)$ denote the maximum and minimum eigenvalues of a symmetric matrix $A$, respectively.

        From $w_k^\top \Sigma_k^{-1} w_k \geq \lambda_{\textup{min}}(\Sigma_k^{-1})\|w_k\|^2$
        and \eqref{eq:Gaussian_support}, the support of the distribution of $w_k$ satisfies:
        \begin{equation}
            \supp{\varphi_{w_k}}\subset 
            \left\{w_k\middle|\ \|w_k\|^2 \leq \frac{n+4-(n+2)q}{(1-q)\lambda_{\textup{min}}(\Sigma_k^{-1})} \right\}.
        \end{equation}
        Therefore, assuming $x_k$ satisfies \eqref{eq:qLQR-supp}, the following holds:
        \begin{align}
            &\|x_{k+1}-\Bar{x}_{k+1}\|^2 \\
            &\quad \leq \|G_k(x_k-\Bar{x}_k)\|^2 + \|B_k w_k\|^2\\
            &\quad \leq \lambda_{\textup{max}}(G_k^\top G_k)\|x_k-\Bar{x}_k\|^2
             + \lambda_{\textup{max}}(B_k^\top B_k)\|w_k\|^2\\
            &\quad \leq \lambda_{\textup{max}}(G_k^\top G_k)r_k^2
            + \lambda_{\textup{max}}(B_k^\top B_k)\frac{n+4-(n+2)q}{(1-q)\lambda_{\textup{min}}(\Sigma_k^{-1})},
        \end{align}
        where $G_k:=A_k + B_k K_k$.
        Since it is trivial that $x_0$ satisfies \eqref{eq:qLQR-supp}, the proof is complete.
    \end{IEEEproof}
\end{document}

%

%

\begin{IEEEbiography}{Michael Shell}
Biography text here.
\end{IEEEbiography}

\begin{IEEEbiographynophoto}{John Doe}
Biography text here.
\end{IEEEbiographynophoto}


\begin{IEEEbiographynophoto}{Jane Doe}
Biography text here.
\end{IEEEbiographynophoto}




\end{document}

%% file: preambles.tex
\usepackage{epsfig} 
\usepackage{amsfonts}
\usepackage{amsmath}
\usepackage{amsthm}
\usepackage{color}

\usepackage{bm}

\usepackage{textcomp}
\usepackage{mleftright}
\usepackage{breqn}
\usepackage{autobreak}

\usepackage{subcaption}

\usepackage[caption=false, font=footnotesize]{subfig}

\allowdisplaybreaks[1]

%% file: macros.tex
\newtheorem{problem}{Problem}
\newtheorem{proposition}{Proposition}
\newtheorem{definition}{Definition}
\newtheorem{lemma}{Lemma}
\newtheorem{theorem}{Theorem}

\newtheorem{rmk}{Remark}

\newcommand{\nc}{\newcommand}
\nc{\disp}{\displaystyle}
\nc{\argmax}{\mathop{\rm arg~max}\limits}
\nc{\argmin}{\mathop{\rm arg~min}\limits}
\nc{\lr}[1]{\ensuremath{\left(#1\right)}}
\nc{\mlr}[1]{\mleft(#1\mright)}
\nc{\alr}[1]{\ensuremath{\left\langle#1\right\rangle}} 
\nc{\blr}[1]{\ensuremath{\left[#1\right]}} 
\nc{\clr}[1]{\ensuremath{\left\{#1\right\}}} 
\nc{\nlr}[1]{\ensuremath{\left\|#1\right\|}}
\nc{\vlr}[1]{\ensuremath{\left|#1\right|}}
\nc{\amlr}[1]{\mleft\langle#1\mright\rangle} 
\nc{\bmlr}[1]{\mleft[#1\mright]} 
\nc{\cmlr}[1]{\mleft\{#1\mright\}} 
\nc{\nmlr}[1]{\mleft\|#1\mright\|}
\nc{\fnmlr}[1]{\mleft\|#1\mright\|_\mathcal{F}}
\nc{\vmlr}[1]{\mleft|#1\mright|}

\nc{\rmvec}{\mathop{\rm vec}}
\nc{\rmtr}{\mathop{\rm tr}}
\nc{\tp}{\mathsf{T}}
\nc{\ex}{\mathbb{E}}
\nc{\normal}{\mathcal{N}}
\nc{\defeq}{:=}
\nc{\cpad}{~}
\nc{\setdef}[2]{\clr{#1 \mathrel{}\middle|\mathrel{} #2}}
\nc{\Z}{\mathbb{Z}}
\nc{\R}{\mathbb{R}}
\nc{\definite}[1]{\mathbb{S}_{#1}^{++}}
\nc{\semidefinite}[1]{\mathbb{S}_{#1}^{+}}
\nc{\N}{\mathbb{N}}
\nc{\ra}{\rightarrow}

\nc{\calA}{{\mathcal A}}    
\nc{\calB}{{\mathcal B}}    
\nc{\calC}{{\mathcal C}}    
\nc{\calD}{{\mathcal D}}    
\nc{\calE}{{\mathcal E}}    
\nc{\calF}{{\mathcal F}}    
\nc{\calG}{{\mathcal G}}    
\nc{\calH}{{\mathcal H}}    
\nc{\calI}{{\mathcal I}}    
\nc{\calJ}{{\mathcal J}}    
\nc{\calK}{{\mathcal K}}    
\nc{\calL}{{\mathcal L}}    
\nc{\calM}{{\mathcal M}}    
\nc{\calN}{{\mathcal N}}    
\nc{\calO}{{\mathcal O}}    
\nc{\calP}{{\mathcal P}}    
\nc{\calQ}{{\mathcal Q}}    
\nc{\calR}{{\mathcal R}}    
\nc{\calS}{{\mathcal S}}    
\nc{\calT}{{\mathcal T}}    
\nc{\calU}{{\mathcal U}}    
\nc{\calV}{{\mathcal V}}    
\nc{\calW}{{\mathcal W}}    
\nc{\calX}{{\mathcal X}}    
\nc{\calY}{{\mathcal Y}}    
\nc{\calZ}{{\mathcal Z}}    

\nc{\scrA}{{\mathscr A}}    
\nc{\scrB}{{\mathscr B}}    
\nc{\scrC}{{\mathscr C}}    
\nc{\scrD}{{\mathscr D}}    
\nc{\scrE}{{\mathscr F}}    
\nc{\scrF}{{\mathscr F}}    
\nc{\scrG}{{\mathscr G}}    
\nc{\scrH}{{\mathscr H}}    
\nc{\scrI}{{\mathscr I}}    
\nc{\scrJ}{{\mathscr J}}    
\nc{\scrK}{{\mathscr K}}    
\nc{\scrL}{{\mathscr L}}    
\nc{\scrM}{{\mathscr M}}    
\nc{\scrN}{{\mathscr N}}    
\nc{\scrO}{{\mathscr O}}    
\nc{\scrP}{{\mathscr P}}    
\nc{\scrQ}{{\mathscr Q}}    
\nc{\scrR}{{\mathscr R}}    
\nc{\scrS}{{\mathscr S}}    
\nc{\scrT}{{\mathscr T}}    
\nc{\scrU}{{\mathscr U}}    
\nc{\scrV}{{\mathscr V}}    
\nc{\scrW}{{\mathscr W}}    
\nc{\scrX}{{\mathscr X}}    
\nc{\scrY}{{\mathscr Y}}    
\nc{\scrZ}{{\mathscr Z}}    

\nc{\bbA}{{\mathbb A}}    
\nc{\bbB}{{\mathbb B}}    
\nc{\bbC}{{\mathbb C}}    
\nc{\bbD}{{\mathbb D}}    
\nc{\bbE}{{\mathbb E}}    
\nc{\bbF}{{\mathbb F}}    
\nc{\bbG}{{\mathbb G}}    
\nc{\bbH}{{\mathbb H}}    
\nc{\bbI}{{\mathbb I}}    
\nc{\bbJ}{{\mathbb J}}    
\nc{\bbK}{{\mathbb K}}    
\nc{\bbL}{{\mathbb L}}    
\nc{\bbM}{{\mathbb M}}    
\nc{\bbN}{{\mathbb N}}    
\nc{\bbO}{{\mathbb O}}    
\nc{\bbP}{{\mathbb P}}    
\nc{\bbQ}{{\mathbb Q}}    
\nc{\bbR}{{\mathbb R}}    
\nc{\bbS}{{\mathbb S}}    
\nc{\bbT}{{\mathbb T}}    
\nc{\bbU}{{\mathbb U}}    
\nc{\bbV}{{\mathbb V}}    
\nc{\bbW}{{\mathbb W}}    
\nc{\bbX}{{\mathbb X}}    
\nc{\bbY}{{\mathbb Y}}    
\nc{\bbZ}{{\mathbb Z}}    

\makeatletter
\ifdefined\p@enumi{
\renewcommand{\p@enumi}{A}
}
\makeatother